\documentclass[11pt]{article}
\usepackage{amsthm, amsmath, amssymb, amsfonts, url, booktabs, tikz, setspace, fancyhdr, bm}
\usepackage{hyperref}
\usepackage{geometry}
\usepackage{hyperref, enumerate}
\usepackage[shortlabels]{enumitem}
\usepackage[babel]{microtype}
\usepackage[english]{babel}
\usepackage[capitalise]{cleveref}
\usepackage{comment}
\usepackage{bbm}
\usepackage{csquotes}
\usepackage{mathabx}
\usepackage{tikz}
\usepackage{graphicx}
\usepackage{float}
\usepackage{amsmath}

\counterwithin{figure}{section}

\newtheorem{theorem}{Theorem}[section]
\newtheorem{prop}[theorem]{Proposition}
\newtheorem{lemma}[theorem]{Lemma}

\newtheorem{conj}{Conjecture}
\newtheorem{claim}[theorem]{Claim}

\theoremstyle{definition}

\newtheorem*{defn-non}{Definition}

\newtheorem{ques}[theorem]{Question}

\newlist{Case}{enumerate}{2}
\setlist[Case, 1]{%
    label           =   {\bfseries Case \arabic*.},
    labelindent=1em ,labelwidth=0.4cm, labelsep*=1em, leftmargin =!
}
\setlist[Case, 2]{%
    label           =   {\bfseries Subcase \arabic{Casei}.\arabic*.},
    labelindent=-1em ,labelwidth=1.0cm, labelsep*=1em, leftmargin =!
}

\newenvironment{poc}{\begin{proof}[Proof of claim]}{\end{proof}}

\newcommand{\aaa}{\boldsymbol{a}}
\newcommand{\bb}{\boldsymbol{b}}
\newcommand{\cc}{\boldsymbol{c}}
\newcommand{\dd}{\boldsymbol{d}}
\newcommand{\ee}{\boldsymbol{e}}
\usepackage{todonotes}

\title{A new variant of the Erd\H{o}s-Gy\'{a}rf\'{a}s problem on $K_{5}$ }
\author{
Gennian Ge\thanks{School of Mathematical Sciences, Capital Normal University, Beijing, China. Emails: gnge@zju.edu.cn, 1427307795@qq.com. Gennian Ge was supported by the National Key Research and Development Program of China under Grant 2020YFA0712100 and Grant 2018YFA0704703, the National Natural Science Foundation of China under Grant 11971325 and Grant 12231014, and Beijing Scholars Program.}
\and 
Zixiang Xu\thanks{Extremal Combinatorics and Probability Group (ECOPRO), Institute for Basic Science (IBS), Daejeon, South Korea. Email: zixiangxu@ibs.re.kr. Supported by IBS-R029-C4.}
\and
Yixuan Zhang\footnotemark[1]
}

\begin{document}
\maketitle
\begin{abstract}
 Motivated by an extremal problem on graph-codes that links coding theory and graph theory, Alon recently proposed a question aiming to find the smallest number $t$ such that there is an edge coloring of $K_{n}$ by $t$ colors with no copy of given graph $H$ in which every color appears an even number of times. When $H=K_{4}$, the question of whether $n^{o(1)}$ colors are enough, was initially emphasized by Alon. Through modifications to the coloring functions originally designed by Mubayi, and Conlon, Fox, Lee and Sudakov, the question of $K_{4}$ has already been addressed. Expanding on this line of inquiry, we further study this new variant of the generalized Ramsey problem and provide a conclusively affirmative answer to Alon's question concerning $K_{5}$.
\end{abstract}

\section{Introduction}
The symmetric difference is a fundamental operation in set theory and combinatorics. It is defined as the set of elements that are present in either of the two given sets but not in their intersection. Kleitman~\cite{1966Kleitman} established a celebrated result in extremal combinatorics, which determines the maximum cardinality of any subset $A\subseteq [n]$ such that the size of the symmetric difference, determined by the elements in $A$, is bounded by a positive integer. There have been many interesting variants and applications around Kleitman's Theorem in recent years~\cite{2017COdingadversarial,2022neighbor,2018CoolingCode,2023Kneighborly,2017CPCFrankl,2023GaoLiuXu,2020Huang,2003CPCSudakov,2017JCTAHamming}.

Recently, a graph-theoretic version of the restricted symmetric difference problem has also received widespread attention~\cite{2023AlonCodes,2023SIAMAlon}. Formally, Alon~\cite{2023AlonCodes} defined a notion of symmetric difference on graphs and introduced a concept called \emph{graph codes}. More precisely, the symmetric difference of two graphs $G_{1}=(V,E_{1})$ and $G_{2}=(V,E_{2})$ on the same vertex set $V$ is the graph $(V,E_{1}\oplus E_{2})$ where $E_{1}\oplus E_{2}$ is the symmetric difference defined to be the set of all edges that belong to exactly one of the two graphs. Let $\mathcal{H}$ be a family of graphs on the vertex set $[n]$ which is closed under isomorphism. We say that a collection $\mathcal{F}$ of graphs on $[n]$ is an \emph{$\mathcal{H}$-code} if
it contains no two members whose symmetric difference is a graph in $\mathcal{H}$. Then a natural question is studying the maximum possible cardinality of an $\mathcal{H}$-code, denoted by $D_{\mathcal{H}}(n)$. Let $d_{\mathcal{H}}(n):=\frac{D_{\mathcal{H}}(n)}{2^{\binom{n}{2}}}$ be the maximum possible fraction of the total number of graphs on $[n]$ in an $\mathcal{H}$-code. If $\mathcal{H}$ only contains a single graph $H$, we also write them as $D_{H}(n)$ and $d_{H}(n)$ for simplicity. As the symmetric differences determined by a family of graphs on $[n]$ are evidently much more complicated than that determined by a collection of subsets of $[n]$, it appears to be impossible to expect the establishment of concise and elegant conclusions analogous to Kleitman's Theorem. Therefore, an initial goal is to study the asymptotic behaviors of the functions $d_{\mathcal{H}}(n)$ for various families $\mathcal{H}$ of graphs.

\begin{ques}
    For a collection of graphs $\mathcal{H}$, determine the order of growth of $d_{\mathcal{H}}(n)$ when $n\rightarrow\infty$.
\end{ques}
Regarding this question, there have already been some interesting results in the past. For instance, a breakthrough of Ellis, Filmus and Friedgut~\cite{2012JEMSDavidEllis} indicates that if $\mathcal{H}$ consists of all graphs with independence number at most $2$, then $d_{\mathcal{H}}(n)=\frac{1}{8}$ for all $n\ge 3$. Moreover, if $\mathcal{H}$ is the family of all graphs with independence number at most $3$, Berger and Zhao~\cite{2021YufeiZhao} recently proved that $d_{\mathcal{H}}(n)=\frac{1}{64}$ for all $n\ge 4$. Alon, Gujgiczer, K\"{o}rner, Milojevi\'{c} and Simonyi~\cite{2023SIAMAlon} also studied this function for families of graphs with special propeties, including the families of all disconnected graphs, all graphs that are not 2-connected, all non-Hamiltonian graphs and all graphs that contain or do not contain a spanning star.

Very recently, Alon~\cite{2023AlonCodes} conducted a comprehensive investigation of this extremal problem, focusing on various graphs, including cliques, stars $K_{1,t}$, and matchings $M_{1,t}$ consisting of $t$ edges. Several results obtained were asymptotically tight, namely, $d_{K_{1,2k}}(n)=\Theta_{k}(\frac{1}{n^{k}})$ and $d_{M_{1,2k}}(n)=\Theta_{k}(\frac{1}{n^{k}})$. Notably, the study of cliques poses a particular interest as it is closely associated with another problem originally proposed by Gowers~\cite{2009Gowers} in 2009.

\begin{conj}[Gowers~\cite{2009Gowers}]
    For every $\delta>0$, there exists $n$ such that if $\mathcal{A}$ is any collection of at least $\delta\cdot 2^{\binom{n}{2}}$ graphs with vertex set $[n]$, then $\mathcal{A}$ contains two distinct graphs $G$ and $H$ such that $G\subseteq H$ and $H\setminus G$ is a clique.
\end{conj}

Let $\mathcal{K}_{r}$ denote the family of all cliques on at most $r$ vertices. Alon~\cite{2023AlonCodes} also proved that $d_{\mathcal{K}_{r}}(n)\geq\Omega(\frac{1}{n^{r}})$. While for a single clique $K_{r}$,
Alon~\cite{2023AlonCodes} then built a connection between this problem and a new coloring problem, which can be viewed as a variant of the Erd\H{o}s-Gy\'{a}rf\'{a}s problem.
The Erd\H{o}s-Gy\'{a}rf\'{a}s problem itself is a very important generalized Ramsey problem, which asks the minimum number of colors $f(n,p,q)$ needed to color the edges of $K_{n}$ such that every copy of $K_{p}$ receives at least $q$ colors. The function $f(n,p,q)$ was introduced about 50 years ago and first studied by Erd\H{o}s and Gy\'{a}rf\'{a}s~\cite{1997EG}. The authors in~\cite{1997EG} claimed that the most annoying problem among the small cases is the behavior of $f(n,4,3)$, and proved that $f(n,4,3)\le O(\sqrt{n})$. Mubayi~\cite{1998Mubayi43} developed an elegant subset ranking coloring and successfully showed $f(n,4,3)=n^{o(1)}$. Later, the subset ranking coloring became a milestone work in the study of this problem, suitable modifications provide a lot of nice lower bounds on $f(n,p,q)$ for different parameters $p$ and $q$~\cite{2018CPC55,2023CPC6688,2015PLMSCFLS,2000Mubayi54,2004Mubayi44}. In particular, Conlon, Fox, Lee and Sudakov~\cite{2015PLMSCFLS} completely proved that $f(n,p,p-1)=n^{o(1)}$ for any integer $p\ge 4$, which yields that $q=p$ is the smallest value of $q$ for which $f(n,p,q)$ is polynomial in $n$, together with $f(n,p,p)=\Omega(n^{\frac{1}{p-2}})$ proven in~\cite{1997EG}. 

Back to the problem on $\mathcal{H}$-codes, Alon~\cite{2023AlonCodes} originally emphasized that the problem is particularly interesting when $\mathcal{H}=\{K_{4}\}$ because he found that, if one can show the existence of an edge coloring of $K_{n}$ by $n^{o(1)}$ many colors with no copy of $K_{4}$ in which every color appears an even number of times, then $d_{K_{4}}(n)\ge\frac{1}{n^{o(1)}}$. It was pointed out by Hunter and Mubayi that the existence of such coloring function has already been shown by Cameron and Heath~\cite{2023CPC6688} via modifying the coloring functions in~\cite{2015PLMSCFLS,1998Mubayi43} (see, the Concluding remarks and open problems in~\cite{2023AlonCodes}). Then Alon suggested this new variant of the Erd\H{o}s-Gy\'{a}rf\'{a}s problem deserves further study. Formally, let $f(n,H)$ be the smallest number of colors in an edge coloring of $K_{n}$ in which every copy of a given graph $H$ intersects at least one of the color classes by an odd number of edges. In particular, we only need to focus on the case that $e(H)$ is even.

\begin{ques}\label{ques:main}
For a given graph $H$ with $e(H)$ being even, determine the order of growth of $f(n,H)$ when $n\rightarrow\infty$. 
\end{ques}

 In this paper, we further explore Question~\ref{ques:main} and show the following upper bound regarding $K_{5}$ when $n\rightarrow\infty$.

\begin{theorem}\label{thm:K5}
$f(n,K_{5})\le e^{O((\log{n})^{2/3}\log\log{n})}$.
\end{theorem}

By a simple application of the celebrated Lov\'{a}sz Local Lemma~\cite{2016AlonProb,1975Locallemma}, we can show the following general upper bound. 
\begin{theorem}\label{thm:LLL}
    For any positive integer $p\ge 4$ with $p\equiv 0,1\pmod{4}$, we have
    \begin{equation*}
        f(n,K_{p})=O(n^{\frac{4(p-2)}{p(p-1)}}).
    \end{equation*}
\end{theorem}

Note that the exponent $\frac{4(p-2)}{p(p-1)}$ tends to $0$ as $p$ becomes larger, thus we believe that the following conjecture holds.

\begin{conj}\label{conj:cliques}
    For any positive integer $p\ge 4$ with $p\equiv 0,1\pmod{4}$, $f(n,K_{p})=n^{o(1)}$.
\end{conj}

As for lower bounds on $f(n,K_{p})$ with positive integer $p\equiv 0,1\pmod{4}$, if for any $\ell$-coloring of $K_{n}$, there is a monochromatic copy of $K_{p}$, then we can find a copy of colored $K_{p}$ in which the single color appears an even number of times. By the best known results on the multi-color Ramsey numbers on cliques $R(p;\ell)\ge 2^{0.383796(\ell-2)p+\frac{p}{2}+o(p)}$~\cite{2022JCTASawin} (also see~\cite{2021ADVConlonFerber,2021PAMSWigderson}), we have that $f(n,K_{p})\ge\Omega_{p}(\frac{\log{n}}{\log{\log{n}}})$. 

Indeed, we have a better lower bound $f(n,K_{5})\ge\Omega(\log{n})$. To see this, we take advantage of a result of Fox and Sudakov~\cite{2008SIAMFoxSudakov}, who showed that if $n$ is large enough and $k\ge c\log{n}$ with some constant $c>0$, then for any edge coloring of $K_{n}$ with $k$ colors, we can find a copy of $K_{4}$ which is monochromatic or is a $C_{4}$ in one color and a matching in the other color. This immediately gives that $f(n,K_{4})\ge\Omega(\log{n})$. Furthermore, for any $k$-coloring of $K_{n}$, and for any vertex $v$, by pigeonhole principle there is a subset $U\subseteq V(K_{n})$ of size $\frac{n-1}{k}$ such that the colors of $uv$ are identical for all $u\in U$. Applying the above result of Fox and Sudakov on the set $U$, we can find a colored $K_{5}$ in which every color appears an even number of times if $k\ge c\log{(\frac{n-1}{k})}$, which implies that $f(n,K_{5})\ge\Omega(\log{n})$. A slightly weaker lower bound $f(n,K_{5})\ge\Omega(\frac{\log{n}}{\log{\log{\log{n}}}})$ could be obtained from the main result in~\cite{2008CPCMubayi}.

\medskip

\noindent\textbf{Note added.} Heath and Zerbib~\cite{2023Emily} also independently proved that $f(n,K_{5})\le n^{o(1)}$. Moreover, they noted that $f(n,K_{p})=O(n^{\frac{4(p-2)}{p(p-1)}}/(\log{n})^{\frac{4}{p(p-1)}})$, which slightly improves the result in Theorem~\ref{thm:LLL}.

\medskip

\noindent\textbf{Organization.} The rest of this paper is organized as follows. We simply show the general upper bound on $f(n,K_{p})$ in Section~\ref{sec:LLL}. Then we turn to prove Theorem~\ref{thm:K5}. First we introduce the coloring function $\psi$ in details and collect several nice properties of $\psi$ in Section~\ref{sec:ColoringFunction}. After classifying all possible colored $K_{5}$ under coloring function $\psi$, we then formally prove the main result in Section~\ref{sec:FormalProof}. Finally we conclude and discuss some further problems in Section~\ref{sec:Remarks}.\\

{\bf \noindent Notations.} In this paper, usually the vertex sets are $\{0,1\}^{k}$, thus we can view the vertices as binary vectors of length $k$ and we typically use bold typeface to represent the vectors $\boldsymbol{v}\in\{0,1\}^{k}$. Given two different binary vectors of the same length, say $\boldsymbol{v}=(v_{1},v_{2},\ldots,v_{k})$ and $\boldsymbol{u}=(u_{1},u_{2},\ldots,u_{k})$, the order of the two vectors depends on the alphabetic order of the symbols in the first place $i$ where they differ (counting from the beginning of the vectors), then $\boldsymbol{v}\prec\boldsymbol{u}$ if and only if $v_{i}<u_{i}$. The notation $\boldsymbol{v}\preceq\boldsymbol{u}$ means $\boldsymbol{v}=\boldsymbol{u}$ or $\boldsymbol{v}\prec\boldsymbol{u}$. For the sake of clarity of presentation, we omit floor and ceiling signs whenever they are not essential.

\section{General upper bound for $f(n,K_{p})$}\label{sec:LLL}
We will take advantage of the following symmetric version of Lov\'{a}sz Local Lemma~\cite{2016AlonProb}.

\begin{lemma}[Lov\'{a}sz Local Lemma]
    Let $A_{1},A_{2},\ldots,A_{k}$ be a sequence of events such that each event occurs with probability at most $p$ and such that each event is independent of all the other events except for at most $D$ of them. If $ep(D+1)<1$, then with positive probability, none of the events occurs.
\end{lemma}

\begin{proof}[Proof of Theorem~\ref{thm:LLL}]

Let $t=cn^{\frac{4(p-2)}{p(p-1)}}$ with a sufficiently large constant $c>0$. Consider the random $t$-coloring of the edges of $K_{n}$, where each edge receives the color $i\in [t]$ with probability $\frac{1}{t}$ independently. Say that a copy of $K_{p}$ is \emph{bad} if the number of edges in each color class is even. Observe that each bad copy of $K_{p}$ receives at most $\frac{\binom{p}{2}}{2}$ colors. For any given $K_{p}$, the probability that this $K_{p}$ is bad is
\begin{equation*}
  \mathbb{P}[K_{p}\ \textup{is\ bad}] \le \sum\limits_{1\le i\le\binom{p}{2}/2\atop \binom{p}{2}/i\ \textup{is\ even }} t^{i}\cdot\bigg(\frac{i}{t}\bigg)^{\binom{p}{2}}\le O(t^{-\frac{p(p-1)}{4}}).
\end{equation*}
Note that every copy of $K_{p}$ is independent of all but at most $D:=\binom{p}{2}\binom{n-2}{p-2}$ many other $K_{p}$, then by Lov\'{a}sz Local Lemma, the result follows since $e\cdot(D+1)\cdot\mathbb{P}[K_{p}\ \textup{is\ bad}]<1$.

\end{proof}

\section{Coloring functions}\label{sec:ColoringFunction}
\subsection{Description of the coloring functions}

 We will employ a simpler and slightly different version of the edge-coloring function $\psi$ of $K_{n}$ compared with that in~\cite{2023CPC6688}. Indeed, our coloring function is also a modification of the coloring functions of Conlon-Fox-Lee-Sudakov~\cite{2015PLMSCFLS} and Mubayi~\cite{1998Mubayi43}. In particular, this coloring inherits all of the previously known properties. Next we describe it in details and introduce some properties of this coloring.

Let $\beta$ be a sufficiently large integer and $\alpha=\beta^{3}$. For $d=\{0,1,2,3\}$, let $r_{d}=\beta^{d}$ and $a_{d}=\beta^{3-d}$. The vertex set of $K_{n}$ is defined to be $\left\{0,1\right\}^\alpha$, and each vertex of $K_n$ can be presented as a binary vector of length $\alpha$. Then for each vector $\boldsymbol{v}\in \left\{0,1 \right\}^{\alpha}$ and $d\in\{0,1,2,3\}$, we can partition this vector into $a_{d}$ blocks and write $\boldsymbol{v}$ as 
\begin{equation*}
    \boldsymbol{v}=\big(\boldsymbol{v}_{1}^{(d)},\boldsymbol{v}_2^{(d)},\ldots,\boldsymbol{v}_{a_{d}}^{(d)}\big),
\end{equation*}
where $\boldsymbol{v}_{i}^{(d)}\in\{0,1\}^{r_{d}}$ for each $i=1,2,\ldots,a_{d}$. In particular when $d=0$, note that $r_{0}=1$ implies that $b_{0}=0$, thus $(\boldsymbol{v}_{1}^{(0)},\boldsymbol{v}_2^{(0)},\ldots,\boldsymbol{v}_{\alpha}^{(0)})$ is exactly the normal binary representation of $\boldsymbol{v}$ and $\boldsymbol{v}_{i}^{(0)}\in\{0,1\}$ for any $1\le i\le\alpha$, that is, $\boldsymbol{v}=(v_{1},v_{2},\ldots,v_{\alpha})$. Moreover, if $d=3$, then $\boldsymbol{v}=\boldsymbol{v}_{1}^{(3)}$.

Next we define several auxiliary families of functions. Let $\boldsymbol{v}:=\{v_{1},v_{2},\ldots,v_{\alpha}\}$ and $\boldsymbol{w}:=\{w_{1},w_{2},\ldots,w_{\alpha}\}$ be two vertices in $\{0,1\}^{\alpha}$. For any $d\in\{0,1,2\}$, we define the functions $\eta_{d}$ as
\begin{center}
    $\eta_{d}(\boldsymbol{v},\boldsymbol{w})=\left\{\begin{matrix}
 \bigg(i,\left\{\boldsymbol{v}_{i}^{(d)},\boldsymbol{w}_i^{(d)} \right\}\bigg)&\textup{if } &\boldsymbol{v}\neq\boldsymbol{w} \\
 0 & \textup{if } &\boldsymbol{v}=\boldsymbol{w} \\
\end{matrix}\right.$,
\end{center}
where $i$ is the minimum index such that $\boldsymbol{v}_i^{(d)}\neq \boldsymbol{w}_i^{(d)}$. Furthermore, for $\boldsymbol{v},\boldsymbol{w}\in \left\{0,1 \right\}^\alpha$ and for each $d\in\{0,1,2\}$, define
\begin{equation*}
    \xi_{d}(\boldsymbol{v},\boldsymbol{w})=\bigg(\eta_{d}(\boldsymbol{v}_{1}^{(d+1)},\boldsymbol{w}_{1}^{(d+1)}),\ldots,\eta_{d}(\boldsymbol{v}_{a_{d+1}}^{(d+1)},\boldsymbol{w}_{a_{d+1}}^{(d+1)})\bigg).
\end{equation*}
Then let $c(\boldsymbol{v},\boldsymbol{w})=(\xi_{2}(\boldsymbol{v},\boldsymbol{w}),\xi_{1}(\boldsymbol{v},\boldsymbol{w}),\xi_{0}(\boldsymbol{v},\boldsymbol{w}))$. In particular, we will repeatedly take advantage of $\xi_{2}(\boldsymbol{v},\boldsymbol{w})=\eta_{2}(\boldsymbol{v},\boldsymbol{w})$ in the proof of our main result.

We then assume the vectors in $\{0,1\}^{\beta}$ are lexicographically ordered. For an ordered pair of vertices $\boldsymbol{v}\preceq \boldsymbol{w}$ and $i\in[\beta]$, define
\begin{center}
    $\delta_{i}(\boldsymbol{v},\boldsymbol{w})=\left\{\begin{matrix}
 +1&if &\boldsymbol{v}_{i}^{(2)}\prec \boldsymbol{w}_{i}^{(2)} \\
 0&if&\boldsymbol{v}_{i}^{(2)} = \boldsymbol{w}_{i}^{(2)} \\
 -1& if &\boldsymbol{v}_{i}^{(2)}\succ\boldsymbol{w}_{i}^{(2)} \\
\end{matrix}\right.$.
\end{center}

We further define
\begin{center}
    $\Delta(\boldsymbol{v},\boldsymbol{w})=\big(\delta_{1}(\boldsymbol{v},\boldsymbol{w}),\ldots,\delta_{\beta}(\boldsymbol{v},\boldsymbol{w})\big)$.
\end{center}

Finally, let
\begin{equation*}
   \psi(\boldsymbol{v},\boldsymbol{w})=(c(\boldsymbol{v},\boldsymbol{w}),\Delta(\boldsymbol{v},\boldsymbol{w})). 
\end{equation*}

Therefore, we can see the number of colors used in function $\Delta$ is at most $3^{\beta}$. The number of colors used in $c$ is at most $\le e^{O(\beta^{2}\log{\beta})}$, as shown in~\cite{2015PLMSCFLS}. Thus the total number of colors we used is $e^{O(\beta^{2}\log{\beta})}=n^{o(1)}$ as $n=e^{\Omega(\beta^{3})}$.

\subsection{Properties of the coloring functions}

Under the coloring function $\psi$, we say that a copy of $K_{4}=\{\aaa,\bb,\cc,\dd\}\subseteq V(G)$ is a striped $K_{4}$, for which (up to isomorphism) $\psi(\aaa\bb)=\psi(\cc\dd)$, $\psi(\aaa\cc)=\psi(\bb\dd)$, $\psi(\aaa\dd)=\psi(\bb\cc)$, $\psi(\aaa\bb)\neq \psi(\aaa\cc)$, $\psi(\aaa\bb)\neq \psi(\aaa\dd)$ and $\psi(\aaa\cc)\neq \psi(\aaa\dd)$.
We will take advantage of the following properties already proven in~\cite{2018CPC55} and~\cite{2023CPC6688}. Let $V(K_{5}):=\{\aaa,\bb,\cc,\dd,\ee\}$. Let $H$ be a graph whose edges are colored by $s$ colors, namely $[s]=\{1,2,\ldots,s\}$, we say $H$ is of \emph{color type} $(a_{1},a_{2},\ldots,a_{s})$, if each color $i$ appears exactly $a_{i}$ times.

\begin{prop}[\cite{2023CPC6688}]\label{prop:Key}\
Under coloring function $\psi$, the followings hold.
    \begin{enumerate}
       \item [\textup{(1)}] Any colored $K_{5}$ (up to isomorphism) with one of the following properties are forbidden:
       
       \begin{itemize}
        \item $\psi(\aaa,\bb)=\psi(\cc,\dd)$ and $\psi(\aaa,\cc)=\psi(\aaa,\dd)$.
        \item $\psi(\aaa,\bb)=\psi(\bb,\cc)=\psi(\cc,\dd)$ and $\psi(\aaa,\cc)=\psi(\cc,\ee)=\psi(\dd,\ee)$.
        \item $\psi(\aaa,\bb)=\psi(\aaa,\ee)=\psi(\cc,\ee)$ and $\psi(\aaa,\dd)=\psi(\dd,\ee)=\psi(\bb,\cc)$.
       \end{itemize}
      
        \item[\textup{(2)}] There is no striped $K_{4}$ under the edge-coloring $\psi$. Moreover, there is no $K_{4}$ of color type $(2,2,2)$.
        \item[\textup{(3)}] Every copy of $K_{5}$ receives at least $4$ distinct colors.
        \item[\textup{(4)}] There is no monochromatic odd cycle of any length.
    \end{enumerate}
\end{prop}
\begin{figure}[bhtp]
    \centering
    \includegraphics[width=12cm]{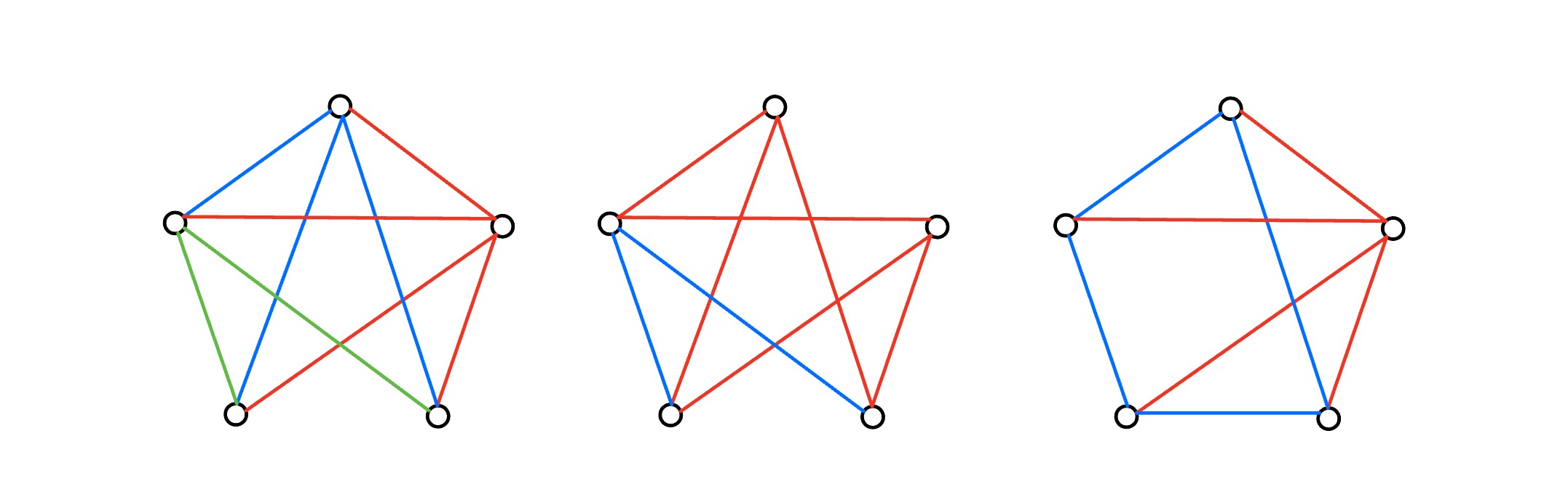}
    \caption{Possible colored $K_{5}$ (up to isomorphism) with at most $4$ colors under coloring $\psi$, where the colors of the missing edges can be chosen arbitrarily}
    \label{fig:4colored}
\end{figure}

\subsection{All possible colored $K_{5}$ of color type $(2,2,2,2,2)$}
By Proposition~\ref{prop:Key}~(3), every copy of $K_{5}$ receives at least $4$ colors, then the possible color types are $(2,2,2,4)$ and $(2,2,2,2,2)$. 
By the argument in~\cite{2018CPC55}, under the coloring function $\psi$, all of the distinct colored $K_{5}$ (up to isomorphism) with at most $4$ colors can be listed in Figure~\ref{fig:4colored} (Also, see, Figure~3 in~\cite{2018CPC55}). Obviously, none of the colored subgraphs in Figure~\ref{fig:4colored} can induce a colored $K_{5}$ of color type $(2,2,2,4)$.

Then our main goal in this section is to show the following result. 

\begin{figure}[bhtp]
    \centering
    \includegraphics[width=14cm]{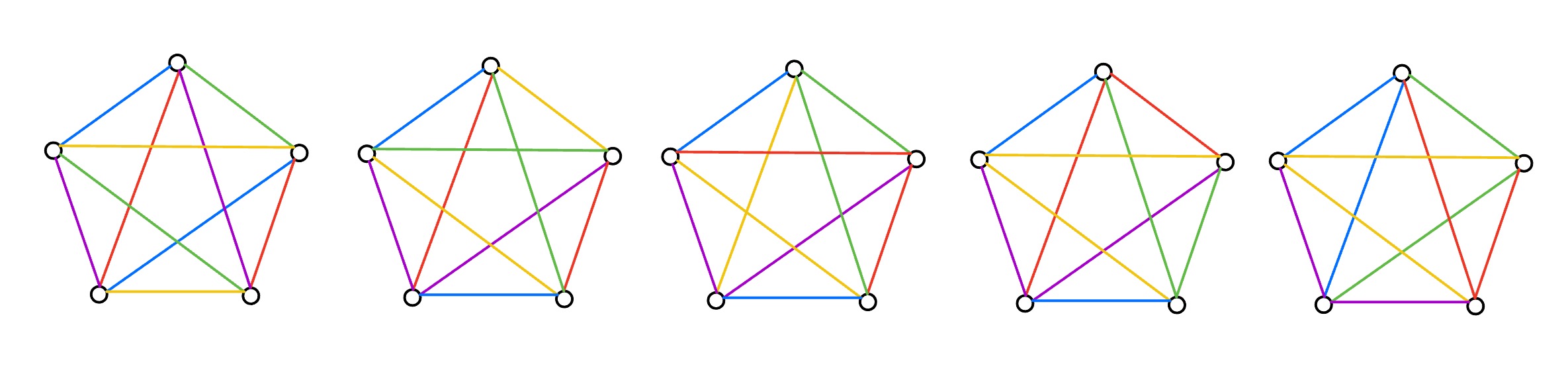}
    \caption{Five possible colored $K_{5}$ of color type $(2,2,2,2,2)$ (up to isomorphism) under coloring $\psi$ }
    \label{fig:possibleK5}
\end{figure}
\begin{prop}\label{Allpossible}
    Under coloring function $\psi$, all of the possible colored $K_{5}$ of color type $(2,2,2,2,2)$ can be listed in Figure~\ref{fig:possibleK5}.
\end{prop}
\begin{proof}[Proof of Proposition~\ref{Allpossible}]
    
For a copy of star $K_{1,2}$, we denote the internal vertex by \emph{core} and each of its leaves by \emph{root}. First, we can consider a copy of $K_{5}$ as a copy of $K_{4}$ together with a \emph{special vertex} which is adjacent to the other four vertices, and obviously there are $5$ special vertices. For our purpose, in a given colored $K_{5}$, we only need to consider those special vertices which appear to be a core vertex of a monochromatic copy of $K_{1,2}$. Then we classify the colored $K_{5}$ according to the number of monochromatic copies of $K_{1,2}$ in them. We will employ the following rules to represent each special vertex in a colored $K_{5}$. Let $V(K_{5})=\{\aaa,\bb,\cc,\dd,\ee\}$, for a special vertex $\boldsymbol{v}\in V(K_{5})$, we associate $\boldsymbol{v}$ with a vector $(s,t)\in \{0,1,2\}\times\{0,1,2\}$, where $s$ is the number of monochromatic copies of $K_{1,2}$ which contains $\boldsymbol{v}$ as a root and $t$ is the number of monochromatic copies of $K_{1,2}$ in the $K_{4}$ induced by subset $V(K_{5})\setminus \{\boldsymbol{v}\}$. Note that for any colored $K_{5}$, suppose that there exist two monochromatic copies of $K_{1,2}$ in some $K_{1,4}$, since there is no copy of $K_{4}$ of type $(2,2,2)$ by Proposition~\ref{prop:Key}, this colored $K_{5}$ cannot be of color type $(2,2,2,2,2)$. Therefore, we can assume that each vertex can appear at most once as a core of some monochromatic copy of $K_{1,2}$. Then we can use $k$ vectors to represent a colored $K_{5}$, where $k$ is the number of monochromatic copies of $K_{1,2}$ in the colored $K_{5}$ and obviously $0\le k\le 5$. 

\begin{claim}\label{claim:s1s1}
    There is no special vertex that can be associated with the vector $(s,1)$ for any $s\in\{0,1,2\}$.
\end{claim}
\begin{poc}
   Suppose that there is some vertex $\boldsymbol{v}$ which is associated with the vector $(s,1)$, we consider the colored $K_{4}$ induced by $V(K_{5})\setminus \{\boldsymbol{v}\}=\{\boldsymbol{u}_{1},\boldsymbol{u}_{2},\boldsymbol{u}_{3},\boldsymbol{u}_{4}\}$. After suitably reorder these four vertices, one can check that the case of $\psi(\boldsymbol{u}_{1},\boldsymbol{u}_{2})=\psi(\boldsymbol{u}_{3},\boldsymbol{u}_{4})$ and $\psi(\boldsymbol{u}_{1},\boldsymbol{u}_{3})=\psi(\boldsymbol{u}_{1},\boldsymbol{u}_{4})$ will occur, which is a contradiction to the second result in Proposition~\ref{prop:Key}~(1). 
\end{poc}

According to Claim~\ref{claim:s1s1}, when the number of monochromatic copies of $K_{1,2}$ belongs to $\{0,1,4,5\}$, one can easily check by hand that the first, the second, the fourth and the last graphs in Figure~\ref{fig:possibleK5} are the only possible cases (up to isomorphism), respectively. In order to be more reader-friendly, we will not provide a very detailed proof for these four simple cases. Below we will carefully analyze the cases that the number of monochromatic copies of $K_{1,2}$ is $2$ or $3$.

When the number of monochromatic copies of $K_{1,2}$ is $2$, then there are two special vertices, namely $\boldsymbol{u}$ and $\boldsymbol{v}$, which are associated with two vectors. Moreover, these two vectors can only be of the forms $(0,1)$ and $(1,0)$, however, by Claim~\ref{claim:s1s1}, the vector $(0,1)$ is forbidden. Then both of them should be $(1,0)$, which implies that some color class of this colored $K_{5}$ will contain at least $3$ colors. Thus those colored $K_{5}$ cannot be of color type $(2,2,2,2,2)$, that means, there is no colored $K_{5}$ of color type $(2,2,2,2,2)$ with $2$ monochromatic copies of $K_{1,2}$.

It remains to consider the case that the number of monochromatic copies of $K_{1,2}$ is $3$, which turns to be slightly complicated. Let $V(K_{5})=\{\boldsymbol{v}_{1},\boldsymbol{v}_{2},\boldsymbol{v}_{3},\boldsymbol{v}_{4},\boldsymbol{v}_{5}\}$ and without loss of generality, we can asuume that the three special vertices are $\boldsymbol{v}_{1}$, $\boldsymbol{v}_{2}$ and $\boldsymbol{v}_{3}$, which are associated with vectors $\boldsymbol{x}_{1}$, $\boldsymbol{x}_{2}$ and $\boldsymbol{x}_{3}$, respectively. Also by Claim~\ref{claim:s1s1}, $\boldsymbol{x}_{1},\boldsymbol{x}_{2},\boldsymbol{x}_{3}\in\{(0,2),(2,0)\}$, as $(1,1)$ is forbidden.

\begin{claim}\label{claim:No11}
There is exactly one of $\boldsymbol{x}_{1}$, $\boldsymbol{x}_{2}$ and $\boldsymbol{x}_{3}$ to be $(2,0)$.
\end{claim}
\begin{poc}
    We first claim that there is at most one of $\boldsymbol{x}_{1}$, $\boldsymbol{x}_{2}$ and $\boldsymbol{x}_{3}$ to be $(2,0)$. To see this, without loss of generality, we can assume $\boldsymbol{x}_{1}=(2,0)$, as $\boldsymbol{v}_{2}$ and $\boldsymbol{v}_{3}$ are also special vertices, we can see that both of the monochromatic copies of $K_{1,2}$ whose cores are $\boldsymbol{v}_{2}$ and $\boldsymbol{v}_{3}$ contain $\boldsymbol{v}_{1}$ as a root. Furthermore, since $\boldsymbol{v}_{1}$, we have $\boldsymbol{v}_{1}\boldsymbol{v}_{4}\boldsymbol{v}_{5}$ together form a copy of monochromatic $K_{1,2}$ with $\boldsymbol{v}_{1}$ being the core vertex, which implies that there are at most one monochromatic copy of $K_{1,2}$ containing $\boldsymbol{v}_{2}$ as a root and also at most one monochromatic copy of $K_{1,2}$ containing $\boldsymbol{v}_{3}$ as a root. However $\boldsymbol{x}_{2},\boldsymbol{x}_{3}\in\{(2,0),(0,2)\}$, thus $\boldsymbol{x}_{2}=\boldsymbol{x}_{3}=(0,2)$. On the other hand, suppose $\boldsymbol{x}_{1}=\boldsymbol{x}_{2}=\boldsymbol{x}_{3}=(0,2)$, in this case, we also have that $\boldsymbol{v}_{1}\boldsymbol{v}_{4}\boldsymbol{v}_{5}$, $\boldsymbol{v}_{2}\boldsymbol{v}_{4}\boldsymbol{v}_{5}$ and $\boldsymbol{v}_{3}\boldsymbol{v}_{4}\boldsymbol{v}_{5}$ form three monochromatic copies of $K_{1,2}$, where $\boldsymbol{v}_{1},\boldsymbol{v}_{2},\boldsymbol{v}_{3}$ are the core vertices respectively. However, then $\psi(\boldsymbol{v}_{1},\boldsymbol{v}_{2})$, $\psi(\boldsymbol{v}_{1},\boldsymbol{v}_{3})$, $\psi(\boldsymbol{v}_{2},\boldsymbol{v}_{3})$ and $\psi(\boldsymbol{v}_{4},\boldsymbol{v}_{5})$ should be pairwise distinct, which yields that this colored $K_{5}$ can not be of color type $(2,2,2,2,2)$. The proof of this claim is finished.
\end{poc}

By Claim~\ref{claim:No11}, we can see any colored $K_{5}$ of color type $(2,2,2,2,2)$ with $3$ monochromatic copies of $K_{1,2}$ should contain the subgraph (up to isomorphism) in Figure~\ref{fig:ThreeK12}. Then one can further consider the colors of the missing edges and easily check that any such colored graphs are isomorphic to the middle graph in Figure~\ref{fig:possibleK5}. 

\begin{figure}[bhtp]
    \centering
    \includegraphics[width=6cm]{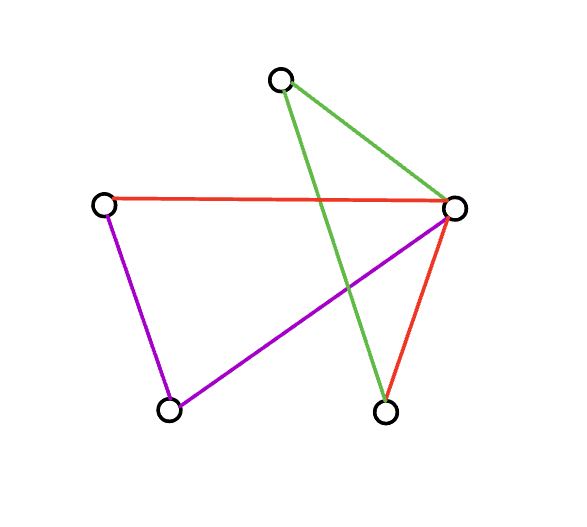}
    \caption{The subgraph contained in any colored $K_{5}$ of color type $(2,2,2,2,2)$ with $3$ monochromatic copies of $K_{1,2}$ (up to isomorphsim)}
    \label{fig:ThreeK12}
\end{figure}

 Based on the detailed analysis so far, we can classify all of the possible colored $K_{5}$ (up to isomorphism) according to the number of monochromatic copies of $K_{1,2}$ in them and we list all of them in Figure~\ref{fig:possibleK5}.
\end{proof}

Therefore, our main goal in the next section, is to prove that all of these five possible colored $K_{5}$ of color type $(2,2,2,2,2)$ cannot exist under the coloring function $\psi$ in details.

\section{Proof of Theorem~\ref{thm:K5}}\label{sec:FormalProof}
Based on the edge-coloring function $\psi$, to show that there is no colored $K_{5}$ such that the number of edges in each color class is even, by Proposition~\ref{Allpossible}, it suffices to prove that there is no $K_{5}$ of color type $(2,2,2,2,2)$ listed in Figure~\ref{fig:possibleK5}. Let $\left\{\aaa,\bb,\cc,\dd,\ee\right\}$ be the vertex set of $K_5$.
 
We divide the whole proof into five parts to show that none of the colored graphs $K_{5}$ (up to isomorphism) in Figure~\ref{fig:possibleK5} exists under the coloring function $\psi$. The beginning two cases (zero and five monochromatic copies of $K_{1,2}$) are relatively simple, and the proofs of remaining cases are complicated, where we need to take advantage of the coloring functions $\delta_{i}$. Recall that $c(\boldsymbol{v},\boldsymbol{w})=(\xi_{2}(\boldsymbol{v},\boldsymbol{w}),\xi_{1}(\boldsymbol{v},\boldsymbol{w}),\xi_{0}(\boldsymbol{v},\boldsymbol{w}))$ and in particular $\xi_{2}(\boldsymbol{v},\boldsymbol{w})=\eta_{2}(\boldsymbol{v},\boldsymbol{w})$. 

We would like to emphasize beforehand that attentive readers may observe the similarity of certain statements in the proofs of various scenarios. However, to ensure a comprehensive demonstration, we have chosen to incorporate as many details as possible. Naturally, in cases of high symmetry, we may omit repetitive descriptions. The primary motivation behind this decision stems from the fact that, indeed, there exist certain discrepancies among different situations.

\subsection{No monochromatic copy of $K_{1,2}$}
Suppose that there is no monochromatic $K_{1,2}$ in the $K_{5}$, by symmetry, we can assume that $\psi(\aaa,\bb)=\psi(\cc,\ee)$, $\psi(\aaa,\cc)=\psi(\dd,\ee)$, $\psi(\aaa,\dd)=\psi(\bb,\cc)$, $\psi(\aaa,\ee)=\psi(\bb,\dd)$ and $\psi(\bb,\ee)=\psi(\cc,\dd)$ and they are pairwise distinct (see, the first graph in Figure~\ref{fig:possibleK5}), which yields that
\begin{itemize}
    \item $\eta_{2}(\aaa,\bb)=\eta_{2}(\cc,\ee)=(i,\left\{\boldsymbol{x}_{1},\boldsymbol{y}_{1}\right\})$;
    \item $\eta_{2}(\aaa,\cc)=\eta_{2}(\dd,\ee)=(j,\left\{\boldsymbol{x}_{2},\boldsymbol{y}_{2}\right\})$;
    \item $\eta_{2}(\aaa,\dd)=\eta_{2}(\bb,\cc)=(k,\left\{\boldsymbol{x}_{3},\boldsymbol{y}_{3}\right\})$;
    \item $\eta_{2}(\aaa,\ee)=\eta_{2}(\bb,\dd)=(\ell,\left\{\boldsymbol{x}_{4},\boldsymbol{y}_{4}\right\})$;
    \item $\eta_{2}(\bb,\ee)=\eta_{2}(\cc,\dd)=(h,\left\{\boldsymbol{x}_{5},\boldsymbol{y}_{5}\right\})$,
\end{itemize}
where $\boldsymbol{a},\bb,\cc,\dd,\ee\in\{0,1\}^{\alpha}$, $\boldsymbol{x}_{s},\boldsymbol{y}_{s}\in\{0,1\}^{r_{2}}$ for $s=1,2,3,4,5$, and $i,j,k,\ell,h\in [a_{2}]$.

Without loss of generality, we can further assume that $i=\min\left\{i,j,k,\ell,h\right\}$, $\boldsymbol{a}_i^{(2)}=\boldsymbol{x}_{1}$ and  
$\boldsymbol{b}_i^{(2)}=\boldsymbol{y}_{1}$. By the definition of the coloring function $\eta_{2}$, it is obvious that $\boldsymbol{x}_{1}\neq\boldsymbol{y}_{1}$. Furthermore, note that $\eta_{2}(\aaa,\bb)=\eta_{2}(\cc,\ee)=(i,\left\{\boldsymbol{x}_{1},\boldsymbol{y}_{1}\right\})$, if $\boldsymbol{c}_i^{(2)}=\boldsymbol{x}_{1}$ and
$\boldsymbol{e}_i^{(2)}=\boldsymbol{y}_{1}$, as $\eta_{2}(\aaa,\ee)=\eta_{2}(\bb,\dd)=(\ell,\left\{\boldsymbol{x}_{4},\boldsymbol{y}_{4}\right\})$ and $\aaa_{i}^{(2)}\neq \ee_{i}^{(2)}$, we have $\ell\le i$. As $i=\min\{i,j,k,\ell,h\}$, we conclude that $i=\ell$. Therefore, $\{\boldsymbol{x}_{1},\boldsymbol{y}_{1}\}=\{\boldsymbol{x}_{4},\boldsymbol{y}_{4}\}$, which yields $\boldsymbol{d}_i^{(2)}=\boldsymbol{x}_{1}$. However, $\boldsymbol{d}_i^{(2)}\neq \boldsymbol{e}_i^{(2)}$ implies that $j=i$ since $\eta_{2}(\dd,\ee)=(j,\left\{\boldsymbol{x}_{2},\boldsymbol{y}_{2}\right\})$. Now we can see that $\eta_{2}(\aaa,\cc)=(j,\left\{\boldsymbol{x}_{2},\boldsymbol{y}_{2}\right\})$, $\{\boldsymbol{x}_{1},\boldsymbol{y}_{1}\}=\{\boldsymbol{x}_{2},\boldsymbol{y}_{2}\}$ and $i=j$ together lead a contradiction to $\aaa_{i}^{(2)}=\cc_{i}^{(2)}=\boldsymbol{x}_{1}$.

On the other hand, if $\boldsymbol{e}_i^{(2)}=\boldsymbol{x}_{1}$ and
$\boldsymbol{c}_i^{(2)}=\boldsymbol{y}_{1}$, similarly, we conclude that $\boldsymbol{d}_i^{(2)}=\boldsymbol{x}_{1}$, as $\eta_{2}(\bb,\ee)=\eta_{2}(\cc,\dd)=(h,\left\{\boldsymbol{x}_{5},\boldsymbol{y}_{5}\right\})$ implies $i=h$ and $\left\{\boldsymbol{x}_{5},\boldsymbol{y}_{5}\right\}=\{\boldsymbol{x}_{1},\boldsymbol{y}_{1}\}$. However, $\boldsymbol{a}_{i}^{(2)}\neq\boldsymbol{c}_{i}^{(2)}$ implies that $j=i$ since $\eta_{2}(\aaa,\cc)=(j,\left\{\boldsymbol{x}_{2},\boldsymbol{y}_{2}\right\})$ and the definition of index $j$. Moreover, $\boldsymbol{d}_{i}^{(2)}=\ee_{i}^{(2)}=\boldsymbol{x}_{1}$ gives a  contradiction to $\dd_{j}^{(2)}\neq\ee_{j}^{(2)}$ and $i=j$.

The case of $\aaa_{i}^{(2)}=\boldsymbol{y}_{1}$ and $\bb_{i}^{(2)}=\boldsymbol{x}_{1}$ can be shown in a very similar way, we omit the repeated argument.

\subsection{The number of monochromatic copies of $K_{1,2}$ is 5}

Suppose that there is a copy of $K_{5}$ that is formed by $5$ edge-disjoint monochromatic copies of $K_{1,2}$, by symmetry we can assume that $\psi(\aaa,\bb)=\psi(\aaa,\cc)$, $\psi(\aaa,\dd)=\psi(\dd,\ee)$, $\psi(\aaa,\ee)=\psi(\ee,\cc)$, $\psi(\bb,\cc)=\psi(\cc,\dd)$ and $\psi(\bb,\ee)=\psi(\bb,\dd)$. 
We also focus on the coloring function $\eta_{2}$ as follows.
\begin{itemize}
    \item $\eta_{2}(\aaa,\bb)=\eta_{2}(\aaa,\cc)=(i,\left\{\boldsymbol{x}_{1},\boldsymbol{y}_{1}\right\})$;
    \item $\eta_{2}(\aaa,\dd)=\eta_{2}(\dd,\ee)=(j,\left\{\boldsymbol{x}_{2},\boldsymbol{y}_{2}\right\})$;
    \item $\eta_{2}(\aaa,\ee)=\eta_{2}(\ee,\cc)=(k,\left\{\boldsymbol{x}_{3},\boldsymbol{y}_{3}\right\})$;
    \item $\eta_{2}(\bb,\cc)=\eta_{2}(\cc,\dd)=(\ell,\left\{\boldsymbol{x}_{4},\boldsymbol{y}_{4}\right\})$;
    \item $\eta_{2}(\bb,\ee)=\eta_{2}(\bb,\dd)=(h,\left\{\boldsymbol{x}_{5},\boldsymbol{y}_{5}\right\})$,
\end{itemize}
where $\boldsymbol{a},\bb,\cc,\dd,\ee\in\{0,1\}^{\alpha}$, $\boldsymbol{x}_{s},\boldsymbol{y}_{s}\in\{0,1\}^{r_{2}}$ for $s=1,2,3,4,5$, and $i,j,k,\ell,h\in [a_{2}]$.

Without loss of generality, we also assume that $i=\min\left\{i,j,k,\ell,h\right\}$, $\boldsymbol{a}_i^{(2)}=\boldsymbol{x}_{1}$ and 
 $\boldsymbol{b}_i^{(2)}=\boldsymbol{c}_i^{(2)}=\boldsymbol{y}_{1}$. Furthermore, note that $\eta_{2}(\bb,\cc)=\eta_{2}(\cc,\dd)=(\ell,\left\{\boldsymbol{x}_{4},\boldsymbol{y}_{4}\right\})$, if $\boldsymbol{d}_i^{(2)}\neq \boldsymbol{y}_{1}$, then we have $\ell\le i$ since $\cc_{i}^{(2)}=\boldsymbol{y}_{1}$ and $\ell$ is the smallest index such that $\cc_{\ell}^{(2)}\neq\dd_{\ell}^{(2)}$. Moreover, $i=\min\{i,j,k,\ell,h\}$ implies $\ell=i$, however $\boldsymbol{b}_{i}^{(2)}=\boldsymbol{c}_{i}^{(2)}$ is impossible by definition of coloring function $\eta_{2}$. Therefore, $\dd_{i}^{(2)}=\boldsymbol{y}_{1}$. Similarly, we claim that $\ee_{i}^{(2)}\neq \boldsymbol{y}_{1}$. To see this, if $\ee_{i}^{(2)}= \boldsymbol{y}_{1}$, then $\aaa_{i}^{(2)}\neq \dd_{i}^{(2)}$ implies that $j=i$, which is a contradiction to $\dd_{i}^{(2)}=\ee_{i}^{(2)}= \boldsymbol{y}_{1}$. Finally, $\ee_{i}^{(2)}\neq\boldsymbol{y}_{1}$ leads to $h=i$ since $\eta_{2}(\bb,\ee)=\eta_{2}(\bb,\dd)=(h,\left\{\boldsymbol{x}_{5},\boldsymbol{y}_{5}\right\})$, which is a contradiction to $\boldsymbol{b}_{i}^{(2)}=\dd_{i}^{(2)}$ by the definition of $\eta_{2}$. Therefore, there is no such colored $K_{5}$, we also omit the analysis of other symmetric cases.

\subsection{The number of monochromatic copies of $K_{1,2}$ is 1}\label{subsubsec:Star1}

 In this case, we can assume that $\psi(\aaa,\bb)=\psi(\cc,\dd)$, $\psi(\bb,\cc)=\psi(\cc,\ee)$, $\psi(\aaa,\dd)=\psi(\bb,\ee)$, $\psi(\aaa,\ee)=\psi(\bb,\dd)$ and $\psi(\aaa,\cc)=\psi(\dd,\ee)$ and they are pairwise distinct (see, the second graph in Figure~\ref{fig:possibleK5}), which yields that

\begin{itemize}
    \item $\eta_{2}(\aaa,\bb)=\eta_{2}(\cc,\dd)=(i,\left\{\boldsymbol{x}_{1},\boldsymbol{y}_{1}\right\})$;
    \item $\eta_{2}(\bb,\cc)=\eta_{2}(\cc,\ee)=(j,\left\{\boldsymbol{x}_{2},\boldsymbol{y}_{2}\right\})$;
    \item $\eta_{2}(\aaa,\ee)=\eta_{2}(\bb,\dd)=(k,\left\{\boldsymbol{x}_{3},\boldsymbol{y}_{3}\right\})$;
    \item $\eta_{2}(\aaa,\dd)=\eta_{2}(\bb,\ee)=(\ell,\left\{\boldsymbol{x}_{4},\boldsymbol{y}_{4}\right\})$;
    \item $\eta_{2}(\aaa,\cc)=\eta_{2}(\dd,\ee)=(h,\left\{\boldsymbol{x}_{5},\boldsymbol{y}_{5}\right\})$,
\end{itemize}
where $\boldsymbol{a},\bb,\cc,\dd,\ee\in\{0,1\}^{\alpha}$, $\boldsymbol{x}_{s},\boldsymbol{y}_{s}\in\{0,1\}^{r_{2}}$, $\boldsymbol{x}_{s}\prec\boldsymbol{y}_{s}$ for $s=1,2,3,4,5$ under lexicographical order, and $i,j,k,\ell,h\in [a_{2}]$.

First by $\eta_{2}(\aaa,\bb)=\eta_{2}(\cc,\dd)=(i,\left\{\boldsymbol{x}_{1},\boldsymbol{y}_{1}\right\})$, without loss of generality, we can assume that $\aaa_{i}^{(2)}=\boldsymbol{x}_{1}$ and $\bb_{i}^{(2)}=\boldsymbol{y}_{1}$, also notice that the case of $\aaa_{i}^{(2)}=\boldsymbol{y}_{1}$ and $\bb_{i}^{(2)}=\boldsymbol{x}_{1}$ can be simlilarly dealt with. Furthermore, regarding the situations of $\cc_{i}^{(2)}$, $\dd_{i}^{(2)}$ and $\ee_{i}^{(2)}$, there are four different cases, so we partition our discussion into several pieces, where the argument in Subcase~2.2 is slightly more complicated. 

\begin{Case}
    \item If $\cc_{i}^{(2)}=\boldsymbol{x}_{1}$ and $\dd_{i}^{(2)}=\boldsymbol{y}_{1}$:  
    \begin{Case}
        \item If $\ee_{i}^{(2)}\neq \boldsymbol{y}_{1}$, as $\bb_{i}^{(2)}\neq \cc_{i}^{(2)}$, we have that $j\le i$ by definition of the indices $i$ and $j$. Furthermore, $\bb_{i}^{(2)}\neq \ee_{i}^{(2)}$ implies that $j\neq i$, giving that $j<i$. By $\eta_{2}(\bb,\cc)=\eta_{2}(\cc,\ee)=(j,\left\{\boldsymbol{x}_{2},\boldsymbol{y}_{2}\right\})$, we know that $\bb_{j}^{(2)}=\ee_{j}^{(2)}\neq\cc_{j}^{(2)}$ then $\boldsymbol{a}_{j}^{(2)}=\bb_{j}^{(2)}$ as $j<i$ and $i$ is the minimum index such that $\boldsymbol{a}_{i}^{(2)}\neq\bb_{i}^{(2)}$. For the same reason, we have $\dd_{j}^{(2)}=\boldsymbol{c}_{j}^{(2)}$. Now we know $\dd_{j}^{(2)}\neq\aaa_{j}^{(2)}$, however, $\bb_{j}^{(2)}=\ee_{j}^{(2)}$, thus $\ell<j$ according to $\eta_{2}(\aaa,\dd)=\eta_{2}(\bb,\ee)=(\ell,\left\{\boldsymbol{x}_{4},\boldsymbol{y}_{4}\right\})$ and the definition of the index $\ell$. Moreover, $\eta_{2}(\bb,\ee)=(\ell,\left\{\boldsymbol{x}_{4},\boldsymbol{y}_{4}\right\})$ implies that $\bb_{\ell}^{(2)}\neq\ee_{\ell}^{(2)}$, thus $\cc_{\ell}^{(2)}$ should be different from at least one of $\bb_{\ell}^{(2)}$ and $\ee_{\ell}^{(2)}$, which yields that $j\le \ell$ since $j$ is the smallest index such that $\bb_{j}^{(2)}=\ee_{j}^{(2)}\neq\cc_{j}^{(2)}$, a contradiction to $\ell<j$.

        \item If $\ee_{i}^{(2)}=\boldsymbol{y}_{1}$, then $\bb_{i}^{(2)}=\ee_{i}^{(2)}\neq\cc_{i}^{(2)}$, as $j$ is the minimum index such that $\bb_{j}^{(2)}\neq\cc_{j}^{(2)}$, we have $j\le i$. Observe that if $j<i$, then we can apply the identical argument to show a contradiction as above (consider $\eta_{2}(\aaa,\dd)=\eta_{2}(\bb,\ee)=(\ell,\left\{\boldsymbol{x}_{4},\boldsymbol{y}_{4}\right\})$ and discuss the relation between $\ell$ and $j$), therefore, it suffices to show that $j=i$ is impossible. If $j=i$, then $\aaa_{i}^{(2)}\neq  \ee_{i}^{(2)}$, and $\bb_{i}^{(2)}=\dd_{i}^{(2)}$. Moreover, as $\eta_{2}(\aaa,\ee)=\eta_{2}(\bb,\dd)=(k,\left\{\boldsymbol{x}_{3},\boldsymbol{y}_{3}\right\})$, we have $k<i$ by definition of indices $i$ and $k$. Similarly, $k<i$ also implies that $\aaa_{k}^{(2)}=\bb_{k}^{(2)}\neq \cc_{k}^{(2)}=\dd_{k}^{(2)}=\ee_{k}^{(2)}$. Observe that $\bb_{k}^{(2)}\neq\cc_{k}^{(2)}=\ee_{k}^{(2)}$, by $\eta_{2}(\bb,\cc)=\eta_{2}(\cc,\ee)=(j,\left\{\boldsymbol{x}_{2},\boldsymbol{y}_{2}\right\})$, we have $j<k$ because $j$ is the minimum index such that $\bb_{j}^{(2)}=\ee_{j}^{(2)}\neq\cc_{j}^{(2)}$, which is a contradiction to $k<i=j$.
    \end{Case}

    \item If $\cc_{i}^{(2)}=\boldsymbol{y}_{1}$ and $\dd_{i}^{(2)}=\boldsymbol{x}_{1}$: 
    \begin{Case}
        \item If $\ee_{i}^{(2)}\neq\boldsymbol{y}_{1}$, then $\ee_{i}^{(2)}\neq\cc_{i}^{(2)}$ and $\bb_{i}^{(2)}=\cc_{i}^{(2)}$, then by $\eta_{2}(\bb,\cc)=\eta_{2}(\cc,\ee)=(j,\left\{\boldsymbol{x}_{2},\boldsymbol{y}_{2}\right\})$ and the definition of indices $i$ and $j$, we have $j<i$, which also yields that $\aaa_{j}^{(2)}=\bb_{j}^{(2)}$ and $\cc_{j}^{(2)}=\dd_{j}^{(2)}$. As $\bb_{j}^{(2)}=\ee_{j}^{(2)}$ and $\aaa_{j}^{(2)}\neq \dd_{j}^{(2)}$, by $\eta_{2}(\bb,\ee)=\eta_{2}(\aaa,\dd)=(\ell,\left\{\boldsymbol{x}_{4},\boldsymbol{y}_{4}\right\})$, as $\ell$ is the smallest index such that $\aaa_{\ell}^{(2)}\neq \dd_{\ell}^{(2)}$, we then have $\ell<j$. However, since $\bb_{\ell}^{(2)}\neq \ee_{\ell}^{(2)}$, $\cc_{\ell}^{(2)}$ must be different from at least one of $\bb_{\ell}^{(2)}$ and $\ee_{\ell}^{(2)}$, which yields that $j\le \ell$, a contradiction to $\ell<j$.
        
        \item If $\ee_{i}^{(2)}=\boldsymbol{y}_{1}$, then $\bb_{i}^{(2)}=\ee_{i}^{(2)}=\cc_{i}^{(2)}$, as $j$ is the minimum index such that $\bb_{j}^{(2)}\neq\cc_{j}^{(2)}$, we have $j\neq i$. Observe that if $j<i$, then we can apply the identical argument to show a contradiction as above (also consider $\eta_{2}(\aaa,\dd)=\eta_{2}(\bb,\ee)=(\ell,\left\{\boldsymbol{x}_{4},\boldsymbol{y}_{4}\right\}$), therefore, we only need to show that $j>i$ is impossible. To see this, suppose $j>i$, note that $\bb_{\ell}^{(2)}\neq \ee_{\ell}^{(2)}$, we have $\cc_{\ell}^{(2)}$ is different from at least one of $\bb_{\ell}^{(2)}$ and $\ee_{\ell}^{(2)}$, which yields that $j<\ell$. Then  $\aaa_{j}^{(2)}=\dd_{j}^{(2)}$, since $\eta_{2}(\aaa,\dd)=\eta_{2}(\bb,\ee)=(\ell,\left\{\boldsymbol{x}_{4},\boldsymbol{y}_{4}\right\}$) and $\ell$ is the smallest index such that $\aaa_{\ell}^{(2)}\neq\dd_{\ell}^{(2)}$. Next we consider $\eta_{2}(\aaa,\cc)=\eta_{2}(\dd,\ee)=(h,\left\{\boldsymbol{x}_{5},\boldsymbol{y}_{5}\right\})$, we claim that $h<j$ always holds. Indeed, note that $\bb_{j}^{(2)}=\ee_{j}^{(2)}\neq\cc_{j}^{(2)}$, suppose that $\aaa_{j}^{(2)}=\cc_{j}^{(2)}$, then $\dd_{j}^{(2)}\neq\ee_{j}^{(2)}$, which implies that $h<j$ by definition of the indices $j$ and $h$. On the other hand, suppose $\aaa_{j}^{(2)}\neq\cc_{j}^{(2)}$, we can see $\cc_{j}^{(2)}\notin\{\dd_{j}^{(2)},\ee_{j}^{(2)}\}$, which implies that $h<j$. Now we have $\ell>j> \max\{i,h\}$.
        
         Suppose $\aaa_{\ell}^{(2)}=\boldsymbol{x}_{4}$ and $\dd_{\ell}^{(2)}=\boldsymbol{y}_{4}$, if $\bb_{\ell}^{(2)}=\boldsymbol{x}_{4}$ and $\ee_{\ell}^{(2)}=\boldsymbol{y}_{4}$, by $\delta_{\ell}(\aaa,\bb)=0$ and $\psi(\aaa,\bb)=\psi(\cc,\dd)$, which yields that $\delta_{\ell}(\cc,\dd)=0$, that means, $\cc_{\ell}^{(2)}=\dd_{\ell}^{(2)}=\boldsymbol{y}_{4}$. However $\delta_{\ell}(\cc,\ee)=0$ and $\delta_{\ell}(\bb,\cc)\neq 0$ together form a contradiction to $\psi(\bb,\cc)=\psi(\cc,\ee)$. Secondly, if $\bb_{\ell}^{(2)}=\boldsymbol{y}_{4}$ and $\ee_{\ell}^{(2)}=\boldsymbol{x}_{4}$, by $\delta_{\ell}(\aaa,\bb)=+1$ and $\psi(\aaa,\bb)=\psi(\cc,\dd)$, we have $\delta_{\ell}(\dd,\cc)=+1$, which indicates that $\dd_{\ell}^{(2)}=\boldsymbol{y}_{4}\prec\cc_{\ell}^{(2)}$. Then we have $\delta_{\ell}(\aaa,\cc)=+1$ and $\delta_{\ell}(\dd,\ee)=-1$, but it leads to a contradiction to $\psi(\aaa,\cc)=\psi(\dd,\ee)$. We can apply the almost identical analysis to show the contradictions under the conditions $\aaa_{\ell}^{(2)}=\ee_{\ell}^{(2)}=\boldsymbol{y}_{4},\bb_{\ell}^{(2)}=\dd_{\ell}^{(2)}=\boldsymbol{x}_{4}$ and $\aaa_{\ell}^{(2)}=\bb_{\ell}^{(2)}=\boldsymbol{y}_{4},\ee_{\ell}^{(2)}=\dd_{\ell}^{(2)}=\boldsymbol{x}_{4}$ respectively, we leave them to the interested readers.
          
    \end{Case}
\end{Case}

\subsection{The number of monochromatic copies of $K_{1,2}$ is 3}\label{subsubsec:star3}
 In this case, we can assume that $\psi(\aaa,\bb)=\psi(\cc,\dd)$, $\psi(\bb,\cc)=\psi(\cc,\ee)$, $\psi(\aaa,\dd)=\psi(\aaa,\ee)$, $\psi(\bb,\ee)=\psi(\dd,\ee)$ and $\psi(\aaa,\cc)=\psi(\bb,\dd)$ and they are pairwise distinct (see, the middle graph in Figure~\ref{fig:possibleK5}), which yields that

\begin{itemize}
    \item $\eta_{2}(\aaa,\bb)=\eta_{2}(\cc,\dd)=(i,\left\{\boldsymbol{x}_{1},\boldsymbol{y}_{1}\right\})$;
    \item $\eta_{2}(\bb,\cc)=\eta_{2}(\cc,\ee)=(j,\left\{\boldsymbol{x}_{2},\boldsymbol{y}_{2}\right\})$;
    \item $\eta_{2}(\aaa,\dd)=\eta_{2}(\aaa,\ee)=(k,\left\{\boldsymbol{x}_{3},\boldsymbol{y}_{3}\right\})$;
    \item $\eta_{2}(\bb,\ee)=\eta_{2}(\dd,\ee)=(\ell,\left\{\boldsymbol{x}_{4},\boldsymbol{y}_{4}\right\})$;
    \item $\eta_{2}(\aaa,\cc)=\eta_{2}(\bb,\dd)=(h,\left\{\boldsymbol{x}_{5},\boldsymbol{y}_{5}\right\})$,
\end{itemize}
where $\boldsymbol{a},\bb,\cc,\dd,\ee\in\{0,1\}^{\alpha}$, $\boldsymbol{x}_{s},\boldsymbol{y}_{s}\in\{0,1\}^{r_{2}}$, $\boldsymbol{x}_{s}\prec\boldsymbol{y}_{s}$ for $s=1,2,3,4,5$ under lexicographical order, and $i,j,k,\ell,h\in [a_{2}]$.

First by $\eta_{2}(\aaa,\bb)=\eta_{2}(\cc,\dd)=(i,\left\{\boldsymbol{x}_{1},\boldsymbol{y}_{1}\right\})$, without loss of generality, we assume that $\aaa_{i}^{(2)}=\boldsymbol{x}_{1}$ and $\bb_{i}^{(2)}=\boldsymbol{y}_{1}$. We also partition the argument into several cases according to the distribution of $\cc_{i}^{(2)}$, $\dd_{i}^{(2)}$ and $\ee_{i}^{(2)}$.

\begin{Case}
    \item If $\cc_{i}^{(2)}=\boldsymbol{x}_{1}$ and $\dd_{i}^{(2)}=\boldsymbol{y}_{1}$: 
    \begin{Case}
        \item If $\ee_{i}^{(2)}\neq \boldsymbol{y}_{1}$, as $\bb_{i}^{(2)}\neq \cc_{i}^{(2)}$, we have that $j\le i$ by definition of indices $i$ and $j$. Furthermore, $\bb_{i}^{(2)}\neq \ee_{i}^{(2)}$ implies that $j\neq i$, giving that $j<i$. By $\eta_{2}(\bb,\cc)=\eta_{2}(\cc,\ee)=(j,\left\{\boldsymbol{x}_{2},\boldsymbol{y}_{2}\right\})$, we know that $\bb_{j}^{(2)}=\ee_{j}^{(2)}\neq\cc_{j}^{(2)}$ then $\boldsymbol{a}_{j}^{(2)}=\bb_{j}^{(2)}$ as $j<i$ and $i$ is the minimum index such that $\boldsymbol{a}_{i}^{(2)}\neq\bb_{i}^{(2)}$. For the same reason, we have $\dd_{j}^{(2)}=\boldsymbol{c}_{j}^{(2)}$. Now we know $\dd_{j}^{(2)}\neq\ee_{j}^{(2)}$, however, $\bb_{j}^{(2)}=\ee_{j}^{(2)}$, thus $\ell<j$. Next, we consider $\eta_{2}(\bb,\ee)=\eta_{2}(\dd,\ee)=(\ell,\left\{\boldsymbol{x}_{4},\boldsymbol{y}_{4}\right\})$, we can see that $\bb_{\ell}^{(2)}=\dd_{\ell}^{(2)}\neq \ee_{\ell}^{(2)}$, since $\ell<j<i$, then $\aaa_{\ell}^{(2)}=\bb_{\ell}^{(2)}$ and $\boldsymbol{c}_{\ell}^{(2)}=\dd_{\ell}^{(2)}$, which indicates that $\aaa_{\ell}^{(2)}=\bb_{\ell}^{(2)}=\boldsymbol{c}_{\ell}^{(2)}=\dd_{\ell}^{(2)}$. However, $\ee_{\ell}^{(2)}\neq\cc_{\ell}^{(2)}$ is a contradiction to the fact that $j$ is the minimum index such that $\boldsymbol{c}_{j}^{(2)}\neq\ee_{j}^{(2)}$.

        \item If $\ee_{i}^{(2)}=\boldsymbol{y}_{1}$, then $\bb_{i}^{(2)}=\ee_{i}^{(2)}\neq\cc_{i}^{(2)}$, as $j$ is the minimum index such that $\bb_{j}^{(2)}\neq\cc_{j}^{(2)}$, we have $j\le i$. Observe that if $j<i$, then we can run the identical argument to show a contradiction as above (consider the relation between $\ell$ and $j$), therefore, we only need to show that $j=i$ is impossible. If $j=i$, since $\aaa_{i}^{(2)}\neq\dd_{i}^{(2)}$ and $\aaa_{i}^{(2)}\neq\ee_{i}^{(2)}$, we have $k\le i=j$. Suppose that $k<i$, by $\eta_{2}(\aaa,\dd)=\eta_{2}(\aaa,\ee)=(k,\left\{\boldsymbol{x}_{3},\boldsymbol{y}_{3}\right\})$, we have $\dd_{k}^{(2)}=\ee_{k}^{(2)}\neq \aaa_{k}^{(2)}$. As $k<i$ and $i$ is the minimum index such that $\aaa_{i}^{(2)}\neq \bb_{i}^{(2)}$ and $\cc_{i}^{(2)}\neq\dd_{i}^{(2)}$, we have $\aaa_{k}^{(2)}=\bb_{k}^{(2)}\neq \cc_{k}^{(2)}=\dd_{k}^{(2)}=\ee_{k}^{(2)}$. However, $j=i$ is the minimum index such that $\bb_{j}^{(2)}\neq\cc_{j}^{(2)}$, which contradicts $\bb_{k}^{(2)}\neq \cc_{k}^{(2)}$ and $k<i=j$.
        
        We then consider the case that $k=i=j$, by $\eta_{2}(\bb,\ee)=\eta_{2}(\dd,\ee)=(\ell,\left\{\boldsymbol{x}_{4},\boldsymbol{y}_{4}\right\})$, we know that $\boldsymbol{b}_{\ell}^{(2)}\neq\ee_{\ell}^{(2)}$, by $\bb_{j}^{(2)}=\ee_{j}^{(2)}$, and $\cc_{\ell}^{(2)}$ is different from at least one of $\bb_{\ell}^{(2)}$ and $\ee_{\ell}^{(2)}$, we have $\ell> j=i=k$. We also consider $\eta_{2}(\aaa,\cc)=\eta_{2}(\bb,\dd)=(h,\left\{\boldsymbol{x}_{5},\boldsymbol{y}_{5}\right\})$, we first show $h>\ell>i=j=k$, to see this, $\bb_{h}^{(2)}\neq \dd_{h}^{(2)}$ and $\bb_{\ell}^{(2)}=\dd_{\ell}^{(2)}$, as $\ee_{h}^{(2)}$ is different from at least one of $\bb_{h}^{(2)}$ and $\dd_{h}^{(2)}$ and $\ell$ is the minimum index such that $\bb_{\ell}^{(2)}=\dd_{\ell}^{(2)}\neq\ee_{\ell}^{(2)}$, we have $h>\ell$.
        
       Now suppose that $\aaa_{h}^{(2)}=\boldsymbol{x}_{5}$ and $\cc_{h}^{(2)}=\boldsymbol{y}_{5}$, we claim that $\bb_{h}^{(2)}=\boldsymbol{x}_{5}$ and $\dd_{h}^{(2)}=\boldsymbol{y}_{5}$. Note that if $\bb_{h}^{(2)}=\boldsymbol{y}_{5}$ and $\dd_{h}^{(2)}=\boldsymbol{x}_{5}$, we can see $\aaa\prec\cc\prec\min\{\bb,\dd,\ee\}$ under lexicographical order, which leads to $\delta_{h}(\aaa,\bb)=+1$ and $\delta_{h}(\cc,\dd)=-1$, a contradiction to $\psi(\aaa\bb)=\psi(\cc\dd)$. 

       Finally, there are in total three possible situations in $\ee_{h}^{(2)}$, that is, $\ee_{h}^{(2)}\preceq\boldsymbol{x}_{5}$, $\boldsymbol{x}_{5}\prec\ee_{h}^{(2)}\prec \boldsymbol{y}_{5}$ and $\boldsymbol{y}_{5}\preceq \ee_{h}^{(2)}$. First, if $\ee_{h}^{(2)}\preceq\boldsymbol{x}_{5}$, then $\delta_{h}(\aaa,\dd)=+1$ and $\delta_{h}(\aaa,\ee)\neq +1$, a contradiction. Second, if $\boldsymbol{x}_{5}\prec\ee_{h}^{(2)}\prec \boldsymbol{y}_{5}$, note that $\ee\prec \boldsymbol{b}$ if and only if $\bb_{\ell}^{(2)}=\boldsymbol{y}_{4}$ and $\ee_{\ell}^{(2)}=\boldsymbol{x}_{4}$ (respectively, $\boldsymbol{b}\prec\ee$ if and only if $\bb_{\ell}^{(2)}=\boldsymbol{x}_{4}$ and $\ee_{\ell}^{(2)}=\boldsymbol{y}_{4}$), then $\delta_{h}(\ee,\bb)=-1$ and $\delta_{h}(\ee,\dd)=+1$, (respectively, $\delta_{h}(\bb,\ee)=+1$ and $\delta_{h}(\dd,\ee)=-1$), a contradiction. Lastly, if $\boldsymbol{y}_{5}\preceq \ee_{h}^{(2)}$, then $\delta_{h}(\cc,\bb)=-1$ and $\delta_{h}(\cc,\ee)\neq -1$, a contradiction. 
    \end{Case}

    \item If $\cc_{i}^{(2)}=\boldsymbol{y}_{1}$ and $\dd_{i}^{(2)}=\boldsymbol{x}_{1}$:
    \begin{Case}
        \item If $\ee_{i}^{(2)}\neq\boldsymbol{y}_{1}$, then $\ee_{i}^{(2)}\neq\cc_{i}^{(2)}$ and $\bb_{i}^{(2)}=\cc_{i}^{(2)}$, then by $\eta_{2}(\bb,\cc)=\eta_{2}(\cc,\ee)=(j,\left\{\boldsymbol{x}_{2},\boldsymbol{y}_{2}\right\})$ and the definition of indices $i$ and $j$, we have $j<i$, which also yields that $\aaa_{j}^{(2)}=\bb_{j}^{(2)}$ and $\cc_{j}^{(2)}=\dd_{j}^{(2)}$. As $\bb_{j}^{(2)}=\ee_{j}^{(2)}\neq\cc_{j}^{(2)}=\dd_{j}^{(2)}$, by $\eta_{2}(\bb,\ee)=\eta_{2}(\dd,\ee)=(\ell,\left\{\boldsymbol{x}_{4},\boldsymbol{y}_{4}\right\})$, as $\ell$ is the smallest index such that $\bb_{\ell}^{(2)}=\dd_{\ell}^{(2)}\neq \ee_{\ell}^{(2)}$, we have $\ell<j$. However, $\cc_{\ell}^{(2)}$ is different from at least one of $\bb_{\ell}^{(2)}$ and $\ee_{\ell}^{(2)}$, which gives $j\le \ell$, a contradiction to $\ell<j$.
        
        \item If $\ee_{i}^{(2)}=\boldsymbol{y}_{1}$, then $\aaa_{i}^{(2)}=\dd_{i}^{(2)}\neq \ee_{i}^{(2)}$, moreover, as $\eta_{2}(\aaa,\dd)=\eta_{2}(\aaa,\ee)=(k,\left\{\boldsymbol{x}_{3},\boldsymbol{y}_{3}\right\})$, $\aaa_{k}^{(2)}\neq\dd_{k}^{(2)}=\ee_{k}^{(2)}$, we have $k<i$ by definition of indices $i$ and $k$. Similarly, $k<i$ also implies that $\aaa_{k}^{(2)}=\bb_{k}^{(2)}\neq \cc_{k}^{(2)}=\dd_{k}^{(2)}=\ee_{k}^{(2)}$. Observe that $\bb_{k}^{(2)}\neq\cc_{k}^{(2)}=\ee_{k}^{(2)}$, by $\eta_{2}(\bb,\cc)=\eta_{2}(\cc,\ee)=(j,\left\{\boldsymbol{x}_{2},\boldsymbol{y}_{2}\right\})$, we have $j<k$ because $j$ is the minimum index such that $\bb_{j}^{(2)}=\ee_{j}^{(2)}\neq\cc_{j}^{(2)}$. Since $j<i$, we have $\aaa_{j}^{(2)}=\bb_{j}^{(2)}=\ee_{j}^{(2)}$ and $\cc_{j}^{(2)}=\dd_{j}^{(2)}$. However, $\aaa_{j}^{(2)}\neq\dd_{j}^{(2)}$ is a contradiction to the fact that $k$ is the minimum index such that $\aaa_{k}^{(2)}\neq\dd_{k}^{(2)}$ and $j<k$.
    \end{Case}
\end{Case}

\subsection{The number of monochromatic copies of $K_{1,2}$ is 4}\label{subsubsec:star4}

In this case, we can assume that $\psi(\aaa,\bb)=\psi(\cc,\dd)$, $\psi(\bb,\cc)=\psi(\cc,\ee)$, $\psi(\aaa,\dd)=\psi(\aaa,\ee)$, $\psi(\bb,\ee)=\psi(\dd,\ee)$ and $\psi(\aaa,\cc)=\psi(\bb,\dd)$ and they are pairwise distinct (see, the fourth graph in Figure~\ref{fig:possibleK5}), which yields that

\begin{itemize}
    \item $\eta_{2}(\aaa,\bb)=\eta_{2}(\cc,\dd)=(i,\left\{\boldsymbol{x}_{1},\boldsymbol{y}_{1}\right\})$;
    \item $\eta_{2}(\bb,\cc)=\eta_{2}(\cc,\ee)=(j,\left\{\boldsymbol{x}_{2},\boldsymbol{y}_{2}\right\})$;
    \item $\eta_{2}(\aaa,\dd)=\eta_{2}(\dd,\ee)=(k,\left\{\boldsymbol{x}_{3},\boldsymbol{y}_{3}\right\})$;
    \item $\eta_{2}(\bb,\dd)=\eta_{2}(\bb,\ee)=(\ell,\left\{\boldsymbol{x}_{4},\boldsymbol{y}_{4}\right\})$;
    \item $\eta_{2}(\aaa,\cc)=\eta_{2}(\aaa,\ee)=(h,\left\{\boldsymbol{x}_{5},\boldsymbol{y}_{5}\right\})$,
\end{itemize}
where $\boldsymbol{a},\bb,\cc,\dd,\ee\in\{0,1\}^{\alpha}$, $\boldsymbol{x}_{s},\boldsymbol{y}_{s}\in\{0,1\}^{r_{2}}$, $\boldsymbol{x}_{s}\prec\boldsymbol{y}_{s}$ for $s=1,2,3,4,5$ under lexicographical order, and $i,j,k,\ell,h\in [a_{2}]$.

First by $\eta_{2}(\aaa,\bb)=\eta_{2}(\cc,\dd)=(i,\left\{\boldsymbol{x}_{1},\boldsymbol{y}_{1}\right\})$, without loss of generality, we assume that $\aaa_{i}^{(2)}=\boldsymbol{x}_{1}$ and $\bb_{i}^{(2)}=\boldsymbol{y}_{1}$. Then it suffices to analyze the following four situations. 

\begin{Case}
    \item If $\cc_{i}^{(2)}=\boldsymbol{x}_{1}$ and $\dd_{i}^{(2)}=\boldsymbol{y}_{1}$:  
    \begin{Case}
        \item If $\ee_{i}^{(2)}\neq \boldsymbol{y}_{1}$, as $\bb_{i}^{(2)}\neq \cc_{i}^{(2)}$, we have that $j\le i$ by definition of indices $i$ and $j$. Furthermore, $\bb_{i}^{(2)}\neq \ee_{i}^{(2)}$ implies that $j\neq i$, giving that $j<i$. By $\eta_{2}(\bb,\cc)=\eta_{2}(\cc,\ee)=(j,\left\{\boldsymbol{x}_{2},\boldsymbol{y}_{2}\right\})$, we know that $\bb_{j}^{(2)}=\ee_{j}^{(2)}\neq\cc_{j}^{(2)}$ then $\boldsymbol{a}_{j}^{(2)}=\bb_{j}^{(2)}$ as $j<i$ and $i$ is the minimum index such that $\boldsymbol{a}_{i}^{(2)}\neq\bb_{i}^{(2)}$. For the same reason, we have $\dd_{j}^{(2)}=\boldsymbol{c}_{j}^{(2)}$. Now we know $\dd_{j}^{(2)}\neq\ee_{j}^{(2)}$, however, $\bb_{j}^{(2)}=\ee_{j}^{(2)}$, thus $\ell<j$ by $\eta_{2}(\bb,\dd)=\eta_{2}(\bb,\ee)=(\ell,\left\{\boldsymbol{x}_{4},\boldsymbol{y}_{4}\right\})$. Furthermore,  we can see that $\bb_{\ell}^{(2)}\neq \ee_{\ell}^{(2)}=\dd_{\ell}^{(2)}$. Since $\ell<j<i$, then $\aaa_{\ell}^{(2)}=\bb_{\ell}^{(2)}$ and $\boldsymbol{c}_{\ell}^{(2)}=\dd_{\ell}^{(2)}$, which indicates that $\aaa_{\ell}^{(2)}=\bb_{\ell}^{(2)}\neq\boldsymbol{c}_{\ell}^{(2)}=\dd_{\ell}^{(2)}$. However, $\bb_{\ell}^{(2)}\neq\cc_{\ell}^{(2)}$ is a contradiction to the fact that $j$ is the minimum index such that $\boldsymbol{c}_{j}^{(2)}\neq\ee_{j}^{(2)}$.

        \item If $\ee_{i}^{(2)}=\boldsymbol{y}_{1}$, then $\bb_{i}^{(2)}=\ee_{i}^{(2)}\neq\cc_{i}^{(2)}$, as $j$ is the minimum index such that $\bb_{j}^{(2)}\neq\cc_{j}^{(2)}$, we have $j\le i$. Observe that if $j<i$, then we can apply the identical analysis to show a contradiction as Subcase 1.1 (argue the relation between $\ell$ and $j$ by $\eta_{2}(\bb,\dd)=\eta_{2}(\bb,\ee)=(\ell,\left\{\boldsymbol{x}_{4},\boldsymbol{y}_{4}\right\})$), therefore, it suffices to show that $j=i$ also leads to some contradiction. Indeed, if $j=i$, as $\eta_{2}(\aaa,\dd)=\eta_{2}(\dd,\ee)=(k,\left\{\boldsymbol{x}_{3},\boldsymbol{y}_{3}\right\})$, $\aaa_{k}^{(2)}=\ee_{k}^{(2)}\neq\dd_{k}^{(2)}$, and $\aaa_{i}^{(2)}\neq \dd_{i}^{(2)}= \ee_{i}^{(2)}$ we have $k<i$ by definition of indices $i$ and $k$. Similarly, $k<i$ also implies that $\aaa_{k}^{(2)}=\bb_{k}^{(2)}=\ee_{k}^{(2)}\neq \cc_{k}^{(2)}=\dd_{k}^{(2)}$. Observe that $\bb_{k}^{(2)}=\ee_{k}^{(2)}\neq\cc_{k}^{(2)}$, by $\eta_{2}(\bb,\cc)=\eta_{2}(\cc,\ee)=(j,\left\{\boldsymbol{x}_{2},\boldsymbol{y}_{2}\right\})$, we have $j\leq k$ because $j$ is the minimum index such that $\bb_{j}^{(2)}=\ee_{j}^{(2)}\neq\cc_{j}^{(2)}$, which is a contradiction to $k<i=j$.
    \end{Case}

    \item If $\cc_{i}^{(2)}=\boldsymbol{y}_{1}$ and $\dd_{i}^{(2)}=\boldsymbol{x}_{1}$: 
    \begin{Case}
        \item If $\ee_{i}^{(2)}\neq\boldsymbol{y}_{1}$, then $\ee_{i}^{(2)}\neq\cc_{i}^{(2)}$ and $\bb_{i}^{(2)}=\cc_{i}^{(2)}$, then by $\eta_{2}(\bb,\cc)=\eta_{2}(\cc,\ee)=(j,\left\{\boldsymbol{x}_{2},\boldsymbol{y}_{2}\right\})$ and the definition of indices $i$ and $j$, we have $j<i$, which also yields that $\aaa_{j}^{(2)}=\bb_{j}^{(2)}$ and $\cc_{j}^{(2)}=\dd_{j}^{(2)}$. As $\bb_{j}^{(2)}=\ee_{j}^{(2)}\neq\cc_{j}^{(2)}=\dd_{j}^{(2)}$, by $\eta_{2}(\bb,\dd)=\eta_{2}(\bb,\ee)=(\ell,\left\{\boldsymbol{x}_{4},\boldsymbol{y}_{4}\right\})$, as $\ell$ is the smallest index such that $\bb_{\ell}^{(2)}\neq \dd_{\ell}^{(2)}= \ee_{\ell}^{(2)}$, we then have $\ell<j$. However, $\bb_{\ell}^{(2)}\neq \ee_{\ell}^{(2)}$ by $\eta_{2}(\bb,\dd)=\eta_{2}(\bb,\ee)=(\ell,\left\{\boldsymbol{x}_{4},\boldsymbol{y}_{4}\right\})$, thus $\cc_{\ell}^{(2)}$ is different from at least one of $\bb_{\ell}^{(2)}$ and $\ee_{\ell}^{(2)}$, which yields that $j\le \ell$, a contradiction to $\ell<j$.
        
        \item If $\ee_{i}^{(2)}=\boldsymbol{y}_{1}$, then $\aaa_{i}^{(2)}=\dd_{i}^{(2)}\neq \ee_{i}^{(2)}$, moreover, as $\eta_{2}(\aaa,\dd)=\eta_{2}(\dd,\ee)=(k,\left\{\boldsymbol{x}_{3},\boldsymbol{y}_{3}\right\})$, $\aaa_{k}^{(2)}=\ee_{k}^{(2)}\neq\dd_{k}^{(2)}$, we have $k<i$ by definition of indices $i$ and $k$. Similarly, $k<i$ also implies that $\aaa_{k}^{(2)}=\bb_{k}^{(2)}=\ee_{k}^{(2)}\neq \cc_{k}^{(2)}=\dd_{k}^{(2)}$. Observe that $\cc_{k}^{(2)}\neq\aaa_{k}^{(2)}=\ee_{k}^{(2)}$, by $\eta_{2}(\aaa,\cc)=\eta_{2}(\aaa,\ee)=(h,\left\{\boldsymbol{x}_{5},\boldsymbol{y}_{5}\right\})$, we have $h<k$ because $h$ is the minimum index such that $\cc_{h}^{(2)}=\ee_{h}^{(2)}\neq\aaa_{j}^{(2)}$. However, $\aaa_{h}^{(2)}\neq\ee_{h}^{(2)}$ by $\eta_{2}(\aaa,\cc)=\eta_{2}(\aaa,\ee)=(h,\left\{\boldsymbol{x}_{5},\boldsymbol{y}_{5}\right\})$, therefore $\dd_{h}^{(2)}$ is different from at least one of $\aaa_{h}^{(2)}$ and $\ee_{h}^{(2)}$, resulting in $k\le h$, a contradiction to $h<k$.
    \end{Case}
\end{Case}

\section{Concluding remarks}\label{sec:Remarks}
Based on the similarity to the classical Erdős-Gyárfás problem and the recent developments~\cite{2015PLMSCFLS,2000Mubayi54,1998Mubayi43}, we are led to speculate that Conjecture~\ref{conj:cliques} holds true. Our investigation has produced supporting evidence, such as Theorems~\ref{thm:K5} and~\ref{thm:LLL}. However, this new conjecture presents additional challenges, as we not only discuss the number of colors assigned to each $K_{p}$, but also need to consider their distribution carefully. For instance, when $p=8$, the potential distribution of colored $8$-cliques becomes notably intricate, surpassing what can be managed manually (although leveraging appropriate computer assistance could be beneficial). Consequently, we think that solely relying on existing approaches is insufficient for proving Conjecture~\ref{conj:cliques}. However, this new problem does not impose a requirement for each $K_{p}$ to receive a large number of colors, opening the possibility of employing an entirely different approach that could prove advantageous.

\bibliographystyle{abbrv}
\bibliography{AlonEG}

\begin{thebibliography}{10}

\bibitem{2017COdingadversarial}
E.~Ahmed and A.~B. Wagner.
\newblock Coding for the large-alphabet adversarial channel.
\newblock {\em IEEE Trans. Inform. Theory}, 63(10):6347--6363, 2017.

\bibitem{2023AlonCodes}
N.~Alon.
\newblock Graph-codes.
\newblock {\em arXiv preprint}, arXiv: 2301.13305, 2023.

\bibitem{2022neighbor}
N.~Alon, J.~Grytczuk, A.~Kisielewicz, and K.~Przesławski.
\newblock New bounds on the maximum number of neighborly boxes in
  {${\mathbb{R}}^{d}$}.
\newblock {\em arXiv preprint}, arXiv: 2212.05133, 2022.

\bibitem{2023SIAMAlon}
N.~Alon, A.~Gujgiczer, J.~K\"{o}rner, A.~Milojevi\'{c}, and G.~Simonyi.
\newblock Structured codes of graphs.
\newblock {\em SIAM J. Discrete Math.}, 37(1):379--403, 2023.

\bibitem{2016AlonProb}
N.~Alon and J.~H. Spencer.
\newblock {\em The probabilistic method}.
\newblock Wiley Series in Discrete Mathematics and Optimization. John Wiley \&
  Sons, Inc., Hoboken, NJ, fourth edition, 2016.

\bibitem{2021YufeiZhao}
A.~Berger and Y.~Zhao.
\newblock {$K_{4}$}-intersecting families of graphs.
\newblock {\em arXiv preprint}, arXiv: 2103.12671, 2021.

\bibitem{2018CPC55}
A.~Cameron and E.~Heath.
\newblock A {$(5,5)$}-colouring of {$K_n$} with few colours.
\newblock {\em Combin. Probab. Comput.}, 27(6):892--912, 2018.

\bibitem{2023CPC6688}
A.~Cameron and E.~Heath.
\newblock New upper bounds for the {E}rd{\H{o}}s-{G}y\'{a}rf\'{a}s problem on
  generalized {R}amsey numbers.
\newblock {\em Combin. Probab. Comput.}, 32(2):349--362, 2023.

\bibitem{2018CoolingCode}
Y.~M. Chee, T.~Etzion, H.~M. Kiah, and A.~Vardy.
\newblock Cooling codes: thermal-management coding for high-performance
  interconnects.
\newblock {\em IEEE Trans. Inform. Theory}, 64(4, part 2):3062--3085, 2018.

\bibitem{2023Kneighborly}
X.~Cheng, M.~Wang, Z.~Xu, and C.~H. Yip.
\newblock Exact values and improved bounds on $k$-neighborly families of boxes.
\newblock {\em arXiv preprint}, arXiv: 2301.06485, 2023.

\bibitem{2021ADVConlonFerber}
D.~Conlon and A.~Ferber.
\newblock Lower bounds for multicolor {R}amsey numbers.
\newblock {\em Adv. Math.}, 378:Paper No. 107528, 5, 2021.

\bibitem{2015PLMSCFLS}
D.~Conlon, J.~Fox, C.~Lee, and B.~Sudakov.
\newblock The {E}rd{\H{o}}s-{G}y\'{a}rf\'{a}s problem on generalized {R}amsey
  numbers.
\newblock {\em Proc. Lond. Math. Soc. (3)}, 110(1):1--18, 2015.

\bibitem{2000Mubayi54}
D.~Eichhorn and D.~Mubayi.
\newblock Edge-coloring cliques with many colors on subcliques.
\newblock {\em Combinatorica}, 20(3):441--444, 2000.

\bibitem{2012JEMSDavidEllis}
D.~Ellis, Y.~Filmus, and E.~Friedgut.
\newblock Triangle-intersecting families of graphs.
\newblock {\em J. Eur. Math. Soc. (JEMS)}, 14(3):841--885, 2012.

\bibitem{1997EG}
P.~Erd\H{o}s and A.~Gy\'{a}rf\'{a}s.
\newblock A variant of the classical {R}amsey problem.
\newblock {\em Combinatorica}, 17(4):459--467, 1997.

\bibitem{1975Locallemma}
P.~Erd\H{o}s and L.~Lov\'{a}sz.
\newblock Problems and results on {$3$}-chromatic hypergraphs and some related
  questions.
\newblock pages 609--627. Colloq. Math. Soc. J\'{a}nos Bolyai, Vol. 10, 1975.

\bibitem{2008SIAMFoxSudakov}
J.~Fox and B.~Sudakov.
\newblock Ramsey-type problem for an almost monochromatic {$K_4$}.
\newblock {\em SIAM J. Discrete Math.}, 23(1):155--162, 2008/09.

\bibitem{2017CPCFrankl}
P.~Frankl.
\newblock A stability result for families with fixed diameter.
\newblock {\em Combin. Probab. Comput.}, 26(4):506--516, 2017.

\bibitem{2023GaoLiuXu}
J.~Gao, H.~Liu, and Z.~Xu.
\newblock Stability through non-shadows.
\newblock {\em Combinatorica, to appear}, arXiv: 2212.07821.

\bibitem{2009Gowers}
T.~Gowers.
\newblock The first unknown case of polynomial {D}{H}{J}.
\newblock
  \url{https://gowers.wordpress.com/2009/11/14/the-first-unknown-case-of-polynomial-dhj/},
  2009.

\bibitem{2023Emily}
E.~Heath and S.~Zerbib.
\newblock Edge-coloring a graph {$G$} so that every copy of a graph {$H$} has
  an odd color class.
\newblock {\em arXiv preprint}, arXiv: 2307.01314, 2023.

\bibitem{2020Huang}
H.~Huang, O.~Klurman, and C.~Pohoata.
\newblock On subsets of the hypercube with prescribed {H}amming distances.
\newblock {\em J. Combin. Theory Ser. A}, 171:105156, 21, 2020.

\bibitem{1966Kleitman}
D.~J. Kleitman.
\newblock On a combinatorial conjecture of {E}rd{\H{o}}s.
\newblock {\em J. Combinatorial Theory}, 1:209--214, 1966.

\bibitem{2008CPCMubayi}
A.~Kostochka and D.~Mubayi.
\newblock When is an almost monochromatic {$K_4$} guaranteed?
\newblock {\em Combin. Probab. Comput.}, 17(6):823--830, 2008.

\bibitem{2003CPCSudakov}
A.~Kostochka and B.~Sudakov.
\newblock On {R}amsey numbers of sparse graphs.
\newblock {\em Combin. Probab. Comput.}, 12(5-6):627--641, 2003.
\newblock Special issue on Ramsey theory.

\bibitem{1998Mubayi43}
D.~Mubayi.
\newblock Edge-coloring cliques with three colors on all {$4$}-cliques.
\newblock {\em Combinatorica}, 18(2):293--296, 1998.

\bibitem{2004Mubayi44}
D.~Mubayi.
\newblock An explicit construction for a {R}amsey problem.
\newblock {\em Combinatorica}, 24(2):313--324, 2004.

\bibitem{2017JCTAHamming}
Y.~Polyanskiy.
\newblock On metric properties of maps between {H}amming spaces and related
  graph homomorphisms.
\newblock {\em J. Combin. Theory Ser. A}, 145:227--251, 2017.

\bibitem{2022JCTASawin}
W.~Sawin.
\newblock An improved lower bound for multicolor {R}amsey numbers and a problem
  of {E}rd{\H{o}}s.
\newblock {\em J. Combin. Theory Ser. A}, 188:Paper No. 105579, 11, 2022.

\bibitem{2021PAMSWigderson}
Y.~Wigderson.
\newblock An improved lower bound on multicolor {R}amsey numbers.
\newblock {\em Proc. Amer. Math. Soc.}, 149(6):2371--2374, 2021.

\end{thebibliography}
\end{document}